\documentclass{article}
\usepackage{commands}
\usepackage[utf8]{inputenc}
\usepackage[margin=1in]{geometry}
\usepackage{biblatex}
\usepackage{todonotes}
\title{Computations associated with the resonance arrangement}
\author{Zachary Chroman, Mihir Singhal} 
\date{April 2021}

\begin{document}

\maketitle

\begin{abstract}
The \textit{resonance arrangement} $\mc A_n$ is the arrangement of hyperplanes in $\R^n$ given by all hyperplanes of the form $\sum_{i \in I} x_i = 0$, where $I$ is a nonempty subset of $\{1,\dots,n\}$. We consider the characteristic polynomial $\chi(\mc A_n; t)$ of the resonance arrangement, whose value $R_n$ at $-1$ is of particular interest, and corresponds to counts of generalized retarded functions in quantum field theory, among other things. No formula is known for either the characteristic polynomial or $R_n$, though $R_n$ has been computed up to $n=8$. By exploiting symmetry and using computational methods, we compute the characteristic polynomial of $\mc A_9$, and thus obtain $R_9$. The coefficients of the characteristic polynomial are also equal to the so-called \textit{Betti numbers} of the complexified hyperplane arrangement; that is, the coefficient of $t^{n-i}$ is denoted by the Betti number $b_i(\mc A_n)$. Explicit formulas are known for the Betti numbers up to $b_3(\mc A_n)$. Using computational methods, we also obtain an explicit formula for $b_4(\mc A_n)$, which gives the $t^{n-4}$ coefficient of the characteristic polynomial.
\end{abstract}

\section{Introduction}

In this paper we are primarily concerned with an arrangement of hyperplanes known as the resonance arrangement:
\begin{defn}
The \textit{resonance arrangement} $\mc A_n$, for integer $n>0$, is the hyperplane arrangement in $\R^n$ (or occasionally $\C^n$) defined by the equations \[\mathcal{A}_n=\left\{\sum_{i \in I}x_i = 0 \mathrel{\Big|} I \subseteq \{1,\dots,n\}\right\}.\] Equivalently, it is the collection of hyperplanes whose normal vectors are nonzero with entries in $\{0,1\}$.
Moreover, we denote by $R_n$ the number of connected regions in the set obtained from $\R^n$ by removing the solution sets of all the equations of $\mc A_n$. 
\end{defn}

The resonance arrangement has also been known as the adjoint of the braid arrangement, or the all-subsets arrangement \cite{Kuhne, KTT}. It has been studied in relation to algebraic geometry, representation theory, and mathematical physics, among other topics. A more complete history of the resonance arrangement can be found in in \cite{Kuhne}. 


To any arrangement of hyperplanes we can associate a \textit{characteristic polynomial}, which is a standard notion in the theory of hyperplane arrangements, given for example in \cite{RS}. 
\begin{defn}
The \textit{characteristic polynomial} of a hyperplane arrangement $\mathcal{A}$ over a field is given by 
\begin{equation} \label{eq:charpoly}
\chi(\mathcal{A};t) = \sum_{T \subseteq \mathcal{A}} (-1)^{|T|} t^{r(\mathcal{A})-r(T)},
\end{equation}
where $r(T)$ denotes the codimension of the intersection of the hyperplanes in $T$.
\end{defn}
It is a theorem of Zaslavsky that $R_n$ is given by the sum of the absolute values of the coefficients of $\chi(\mathcal{A};t)$ (or equivalently, the value $\chi(\mathcal{A};-1)$) \cite{Zaslavsky}. In particular, to compute the sequence $R_n$ it suffices to compute the polynomials $\chi(\mathcal{A}_n;t)$. Unfortunately, these polynomials are considered equally intractable \cite{Kuhne}. 

The characteristic polynomials $\chi(\mathcal{A}_n;t)$ for $n$ at most $7$ were computed in \cite{KTT}. We will compute the next two, and therefore $R_8$ and $R_9$. The polynomial $\chi(\mathcal{A}_9;t)$ has also been computed independently and concurrently by Brysiewicz, Eble, and K\"uhne in \cite{bek}, using different methods. A listing of all the characteristic polynomials is provided in \cref{fig:polys}.

\begin{thm} \label{thm:poly}
The characteristic polynomials of $\mc A_8$ and $\mc A_9$ are given by
\begin{multline*}
\chi(\mathcal{A}_8;t) = t^8-255 t^7+29360 t^6-1957200 t^5+81029004 t^4-2076831708 t^3 \\ {}+30623870732
  t^2-207507589302 t+178881449368,
\end{multline*}
\begin{multline*}
\chi(\mathcal{A}_9;t) = 
t^9-511 t^8+120975 t^7-17153460 t^6+1582492380 t^5-96834110730 t^4\\ {}+3829831100340
   t^3-89702833260450 t^2+973784079284874 t-887815808473419,
\end{multline*}
and their corresponding region counts are given by
\begin{equation*}
R_8 = 419172756930, \qquad R_9 = 1955230985997140.
\end{equation*}
\end{thm}

These results match the independently computed values in \cite[Table 5]{bek}.

In order to show \cref{thm:poly}, we will first show in \cref{sec:bij} that the value of $\chi(\mc A_n; p)$ when $p$ is a sufficiently large prime can be computed by counting a certain type of combinatorial object. We then use computational methods, in addition to exploiting symmetry, to count these objects for various values of $p$ and therefore obtain the polynomials by interpolation. This computation is detailed in \cref{sec:seqs}.

Specific coefficients of the characteristic polynomial are of particular interest. We denote the $t^{n-i}$ coefficient of $\chi(\mathcal{A}_n;t)$ by $b_i(\mathcal{A}_n)$, as these coefficients correspond to the topological Betti numbers of the complexified hyperplane arrangement \cite{Kuhne}. It is shown in \cite{Kuhne} that these $b_i(\mathcal{A}_n)$ for a fixed $i>0$ can be computed as a sum of Stirling numbers of the second kind, which in particular implies they can be written as a linear combination \[b_i(\mathcal{A}_n) = \sum_{j=1}^{2^i} c_{i,j} j^n.\]  The following formulas for $b_i(\mathcal{A}_n)$ for $0 \leq i \leq 3$ and $n>0$ are given in \cite{Kuhne}:

\begin{equation} \label{eq:0-3}
    \begin{aligned}
    &b_0(\mathcal{A}_n) = 1 \\
    &b_1(\mathcal{A}_n) = 2^n-1 \\ 
    &b_2(\mathcal{A}_n) = \frac 12 (4^n-3^n-2^n+1)\\
    &b_3(\mathcal{A}_n) = \frac{1}{24}\left(4 \cdot 8^n-15 \cdot 6^n + 15 \cdot 5^n -14 \cdot 4^n +18 \cdot 3^n -7 \cdot 2^n -1\right).
    \end{aligned}
\end{equation}

Using computational methods, we extend these results to compute a formula for $b_4(\mathcal{A}_n)$:
\begin{thm}
\label{thm:b4}
For all $n \ge 0$, \begin{multline*}
b_4(\mc A_n) = \f{1}{48} (2 \cdot 16^n - 18 \cdot 12^n + 31 \cdot 10^n + 33 \cdot 9^n - 119 \cdot 8^n + 77 \cdot 7^n \\ + 77 \cdot 6^n - 173 \cdot 5^n + 155 \cdot 4^n - 113 \cdot 3^n + 64 \cdot 2^n - 16).
\end{multline*}
\end{thm}
The proof of this result is in \cref{sec:b4}, where we define matrices corresponding to sets of hyperplanes, and again exploit symmetry in order to use computational techniques to effectively compute the coefficients $c_{4,j}$.
    
\section{Combinatorial interpretation of the characteristic polynomial} \label{sec:bij}
We will approach \cref{thm:poly} by computing the values of $\chi(\mc A_8;t)$ and $\chi(\mc A_9;t)$ at certain primes using combinatorial methods, and interpolating. To perform these computations, we first define a related object. 
\begin{defn}
We define a \textit{$(p,d)$-sequence} to be any element of $\F_p^d$ for which no nonempty subset sums to zero. We let $G(p,d)$ denote the set of such sequences.
\end{defn}
The connection between $G(p,d)$ and $\mathcal{A}_n$ is made apparent by the following relationship:
{
\begin{thm}
For any fixed $d$, $|G(p,d)| = \chi(\mathcal{A}_d;p)$ for all $p$ such that no 0--1 matrix of size $d$ has determinant equal to a nonzero multiple of $p$.
\label{thm:G=chi}
\end{thm}
\begin{proof}
For each subset $S \subset \{0,1\}^{[d]}$, let $f(S)$ denote the number of $d$-tuples $v$ which satisfy $v \cdot x = 0$ for each vector $x \in S$; that is, each designated subset sum is zero. We want to find $f(\{0,1\}^{[d]} \setminus \{(0, \dots, 0)\})$. By the Principle of Inclusion-Exclusion, this is the sum over all subsets of $k$ hyperplanes of $(-1)^k$ times the number of points in the intersection of the $k$ hyperplanes, over all $k \leq d$. 

The number of points in the intersection of hyperplanes is given by $n^{d-r}$, where $r$ is the rank over $\F_p$ of the 0--1 matrix $A$ with rows given by the coordinate vectors of the hyperplanes. We claim that $r$ is also the rank of $A$ over $\Q$. Indeed, since all determinants of minors are $0$ mod $p$ if and only if they are $0$ in $\Z$, the rank must be the same. Thus there are $p^{d-r}$ solutions, and therefore $|G(p,d)|$ is a polynomial in $p$, given by \[\sum_{S \subseteq \mathcal{A}_d}p^{d-r(S)}.\] This is exactly the definition of $\chi(\mathcal{A}_d;p)$. 

\end{proof}}


As seen in \cite{103}, for $d=9$ the condition of \cref{thm:G=chi} holds for all $p$ at least $103$. Having established the four leading coefficients in \eqref{eq:0-3}, it suffices to compute six more coefficients of $\chi(\mathcal{A}_9;t)$ and five more of $\chi(\mathcal{A}_8;t)$. Thus it will suffice to compute $|G(p,9)|$ for $6$ primes which are at least $103$. Note that it is possible for a determinant of a 9 by 9 0--1 matrix to be as large as 144, but it will not be divisible by any prime number which is at least 103. In fact, our approach would likely become infeasible if 144 were the best possible lower bound.

In the next section we perform this computation for $p \in \{103,107,109,113,127,131\}$. We ran a similar computation starting at $p=41$ (per \cite{103}) to compute $\chi(\mathcal{A}_8;t)$.

\section{Counting the \texorpdfstring{$(p, d)$}{(p, d)}-sequences} \label{sec:seqs}
In this section we describe the algorithm used to count $|G(p, d)|$, the number of $(p, d)$-sequences, for prime $p$. For each sequence $A$, let $c(A)$ be the count of the most frequent element of $A$, and let $n_k(A)$ be the number of elements in $A$ which occur $k$ times.

We will use recursion to first do casework on the number of 1s in the sequence, then the number of 2s, and so on. By itself, this algorithm would not run quickly enough. In order to improve the running time, we will also scale the sequence by a nonzero constant mod $p$ so that $1$ is the most frequent element of the sequence, and count sequences up to this scaling, multiplying the final answer by $p-1$ to account for the constant factor. For each sequence $A$, there are $n_{c(A)}(A)$ possible choices of this constant factor, so we will need to count each scaled sequence $A$ with multiplicity $1/n_{c(A)}(A)$ to account for overcounting (note that the value of $n_{c(A)}(A)$ is invariant under the scaling).

Now, we will define a function which we will be able to evaluate easily. We define the set $f_A(p, d', x)$, where $A$ is a sequence in $\F_p$, and $x \ge 2$, as follows. Let $c$ be the number of occurrences of $1$ in $A$. If $B$ is a sequence in $\F_p$ of length $d'$, we say that it is \textit{nice} if it consists only of elements between $x$ and $p-1$, inclusive, with no more than $c$ of any such element, such that the concatenation $A ^\frown B$ has no subset sums equal to zero. Then, $f_A(p, d', x)$ is the total count of nice sequences $B$, where each sequence is counted with multiplicity $1/n_c(A ^\frown B)$.

Note that we can evaluate the function $f_A(p, d', x)$ recursively. In particular, if there are $i$ copies of $x$ in $B$, then there are $\binom{d'}{i}$ ways to pick the positions of $i$ in $B$, and the number of ways to pick the rest of the elements of $B$ is equal to $f_{A ^\frown (x)^i}(p, d'-i, x+1)$. Note that $i$ is bounded by $c$, and is also bounded by the maximum number $i_\text{max}$ such that $A ^\frown (x)^{i_\text{max}}$ has no zero sums. (Note that $i_\text{max}$ is just the minimum possible value such that $-(i_\text{max}+1)x$ is a subset sum of $A$.) Thus, by casework on possible values of $i$, we have

\[f_A(p, d', x) = \sum_{x=1}^{\min\{d', c, i_\text{max}\}} \binom{d'}{i} f_{A ^\frown (x)^i}(p, d'-i, x+1).\]

Finally, we may recover $|G(p, d)|$ by casework on the number of 1s in the sequence (noting that there are once again $\binom{d}{c}$ ways to pick the positions of the 1s):
\[|G(p, d)| = \sum_{c=1}^d \binom{d}{c} f_{(i)^c}.\]

Now, we will speed up computation by noting that instead of keeping track of $A$, we only need to keep track of $S$, the set of subset sums of $A$, in addition to $c$ and $n_c(A)$. When we append $x$ to $A$, we just replace $S$ with $S \cup (S+x)$. It is easy to check when we have hit the limit $i_\text{max}$, by checking presence of $-x$ in $S$ after $i-1$ copies of $x$ have already been added. Note that $c$ stays constant throughout the recursion, and $n_c(A)$ gets incremented by one whenever $i=c$.

We used Python to compute the value of $|G(p, d)|$ for $d=9$ and $p \in \{103, 107, 109, 113, 127, 131\}$ (In fact, if we use \cref{thm:b4}, we do not need $p=131$, but we compute it for completeness). The code is available here: \url{https://github.com/mihirsinghal/resonance-arrangement/}. Using PyPy 3, it took about a week to run on a personal laptop. The computed values of $|G(p,9)|$ are given in \cref{fig:G9}.
\begin{table}[ht]
\[\begin{array}{|c|c|}\hline
p & |G(p,9)|\\ \hline
 103 & 7222699176517740 \\ \hline
 107 & 12055973486434780 \\ \hline
 109 & 15451254517025520 \\ \hline
 113 & 24997023634629760 \\ \hline
 127 & 116648572972479660 \\ \hline
 131 & 174658244425697500 \\ \hline
\end{array}
\]
\caption{Computed values of $|G(p, 9)|$ for specific values of $p$}
\label{fig:G9}
\end{table}
Combined with the known coefficients given by \eqref{eq:0-3}, we can interpolate to obtain the polynomial $\chi(\mc A_9; t)$. We provide this polynomial (as well as all previous polynomials) in \cref{fig:polys}.

\begin{table}[ht]
    \centering
\[
\begin{array}{|c|l|}\hline
 \chi(\mathcal{A}_0;t) & 1 \\ \hline
 \chi(\mathcal{A}_1;t)  & t-1 \\ \hline
 \chi(\mathcal{A}_2;t)  & t^2-3 t+2 \\ \hline
 \chi(\mathcal{A}_3;t)  & t^3-7 t^2+15 t-9 \\ \hline
 \chi(\mathcal{A}_4;t)  & t^4-15 t^3+80 t^2-170 t+104 \\ \hline
 \chi(\mathcal{A}_5;t)  & t^5-31 t^4+375 t^3-2130 t^2+5270 t-3485 \\ \hline
 \chi(\mathcal{A}_6;t)  & t^6-63 t^5+1652 t^4-22435 t^3+159460 t^2-510524 t+371909 \\ \hline
 \chi(\mathcal{A}_7;t)  & t^7-127 t^6+7035 t^5-215439 t^4+3831835 t^3-37769977 t^2+169824305 t-135677633 \\ \hline
 \multirow{2}{*}{$\chi(\mathcal{A}_8;t)$}  & t^8-255 t^7+29360 t^6-1957200 t^5+81029004 t^4-2076831708 t^3 \\ & \quad {}+30623870732
  t^2-207507589302 t+178881449368 \\ \hline
 \multirow{2}{*}{$\chi(\mathcal{A}_9;t)$}  & t^9-511 t^8+120975 t^7-17153460 t^6+1582492380 t^5-96834110730 t^4\\ & \quad {}+3829831100340
   t^3-89702833260450 t^2+973784079284874 t-887815808473419
    \\ \hline
\end{array}
\]
    \caption{Polynomials $\chi(\mc A_n; t)$ for $n \le 9$}
    \label{fig:polys}
\end{table}

\section{Computation of \texorpdfstring{$b_4(\mc A_n)$}{b4(An)}} \label{sec:b4}

In this section we will compute the coefficient of $t^{n-4}$ in $\chi(\mathcal{A}_n;t)$. Following notation in \cite{Kuhne}, we denote by $b_k(\mathcal{A}_n)$ the coefficient of $t^{n-k}$. Exact formulas for $b_0$ through $b_3$ are computed by \cite{Kuhne} and reproduced in \eqref{eq:0-3}, and we will extend this to $b_4$ (i.e., the case $k=4$).

In \eqref{eq:charpoly}, we see that to compute $b_4(\mathcal{A}_n)$ we need to compute the number of subsets of $\mathcal{A}_n$ whose intersection has rank $k$. Equivalently, we want to find 0--1 matrices $M$ satisfying the following properties: 
\begin{itemize}
    \item $M$ has rank $k$
    \item $M$ has exactly $n$ rows
    \item No two columns of $M$ are identical
    \item No column of $M$ is identically zero
\end{itemize}
Since the columns of these matrices represent hyperplanes, and sets of hyperplanes are considered up to rearrangement, we wish to count each such matrix with multiplicity $\f1{\ell!}$, where $\ell$ is the number of columns in each matrix.

Say a given set $S$ of (distinct) 0--1 vectors each of length $\ell$ is \emph{good} if, for each pair of positions $1\leq i<j\leq \ell$, at least one of the vectors is not the same in the $i$-th and $j$-th position, there is no position $1 \le i \le \ell$ such that every vector is zero in the $i$-th position, and the matrix with $S$ as its rows has rank $k$. 

Any matrix $M$ satisfies the required conditions iff the set of its distinct rows is good. For a good set of size $s$, the number of matrices $M$ is given by 
\[\sum_{i \leq s} (-1)^{|S|-i} i^n \binom si,\] 
by the Principle of Inclusion-Exclusion. Recall that each set $T$ in \eqref{eq:charpoly} corresponds to $\ell!$ matrices $M$. Thus, the coefficent of $t^{n-k}$ is equal to
\begin{equation} \label{eq:coeffsum}
    b_k(\mc A_n) = \sum_{\text{good $S$}}\f{(-1)^{\ell}}{\ell!}\sum_{i \leq |S|} (-1)^{|S|-i} i^n \binom{|S|}i.
\end{equation}
As an example, the only good sets for $k=1$ are $\{(1)\}$ and $\{(0),(1)\}$, so the coefficient of $t^{n-1}$ is \[\frac{(-1)^1}{1!}\left((-1)^{0}i^n\binom 11\right)+\frac{(-1)^1}{1!}\sum_{1\leq i\leq 2} (-1)^{2-i} i^n \binom 2i = -1+2-2^n=1-2^n,\] which is indeed the coefficient of $t^{n-1}$ in the observed polynomials. To verify the formula given for the $t^{n-2}$ coefficient, it suffices to consider all $17$ possible good sets of rank $2$ ($11$ with $\ell=2$ and $6$ with $\ell=3$). 

In general, we claim that there are only finitely many good sets for fixed $k$. To see this, we'll show that if a matrix with no two identical rows and no two identical columns has at least $2^k+1$ rows (or columns), it must have rank at least $k+1$. Such a matrix with rank at most $k$ would have a set of $k$ columns which generate all the rows. But by the pigeonhole principle two rows must be the same on all $k$ of these columns. Since the columns generate, the rows must be the same on every column, which is a contradiction. 

This implies both that $\ell \leq 2^k$ and also that any good set has size at most $2^k$. In particular, there are only finitely many possible good sets. 

We will now describe the algorithm which we use to compute the number of good sets with $r$ vectors of length $\ell$, for every particular pair of values $r, \ell \le 2^k$. Notice that the zero vector can be added to any good set, so it suffices to be able to count only the good sets which do not contain the zero vector. Note that we can let all the vectors in such a good set be the rows of a matrix $Y$; the $r \times \ell$ matrix $Y$ then satisfies the following conditions.
\begin{itemize}
    \item $Y$ has rank $k$.
    \item No two columns of $Y$ are identical, and no two rows of $Y$ are identical.
    \item No row or column of $Y$ is identically zero.
\end{itemize}
Since all elements of a good set are distinct, each good set corresponds to $r!$ such matrices $Y$, so it is enough to count such matrices $Y$. Additionally, note that these conditions are identical when transposing $Y$, so we can assume that $r \ge \ell$.

Note that the conditions on $Y$ are also preserved when permuting the rows and/or columns of $Y$. Thus, we will first generate all possible matrices $Y$ up to row and column permutation (with possible duplicates). We will then generate all matrices that can be obtained by permuting columns of such matrices, but assuming that the rows are in (lexicographically) sorted order. Then, note that the total number of matrices $Y$ can be obtained by just multiplying by $r!$ to account for row order, since all rows are distinct.

Now we describe how to generate the matrices $Y$ up to row and column permutation, specializing to the particular case $k=4$. We can assume, by column permutation, that the matrix $X$ given by the first 4 columns of $Y$ has rank 4. Note that all columns of $Y$ must be in the span of the columns of $X$. Thus, if $X$ were to contain two identical rows, then these rows would be identical in $Y$ as well. This would contradict the second condition on $Y$, so $X$ must have distinct rows.

Therefore, there are only $2^{15}$ possibilities for $X$, up to row permutation, since the rows of $X$ are all distinct nonzero 0--1 vectors of length 4. $X$ must also contain $k$ linearly independent rows. We generate all such sets of linearly independent rows to obtain all possible $X$ which work. We additionally permute the rows of $X$ so that its first 4 rows are linearly independent; let $B$ be the $4 \times 4$ matrix given by these rows. 

Now, note that $Y$ must be obtained from $X$ by adding column vectors which are in the span of the 4 column vectors of $X$, but maintaining that all columns of $Y$ are distinct. We will thus enumerate the 0--1 column vectors $u$ which are in the span of the columns of $X$ (and are not already in $X$). Note that these are exactly the 0--1 vectors which can be written as $Xw$, for a 4-element column vector $w$. Let $v$ be the vector consisting of the top 4 elements of $u$. Then, taking the top 4 rows of the identity $Xw=u$, we have that $Bw=v$. But $B$ is invertible, so we can thus write $w=B^{-1}v$ and thus that $u = XB^{-1} v$. Now, note that $v$ cannot be any of the columns of $B$ since that would make $u$ be a column of $X$, and it also cannot be zero since that would also make $u$ zero. Thus, there are 11 possibilities for the 0--1 vector $v$. We can then plug in all 11 possibilities and check if the resulting $u$ is a 0--1 vector. 

In this way, we can determine all columns $u$ which can be added to $X$. By adding subsets of such columns $u$, we can thus obtain all possibilities for $Y$ (still up to row and column permutation). For all such possibilities $Y$, we can, as previously described, generate all column permutations, such that the rows are always sorted, to obtain all possibilities for $Y$ up to row permutation, with no duplicates (since the rows are kept sorted). Thus we can obtain the number of such matrices $Y$. After adjusting for the assumptions that there are no zero rows and that $r \ge \ell$, we can obtain the total number of good sets for each $r, \ell \le 16$. Finally, we can now use the formula given by \eqref{eq:coeffsum} to compute a formula for $b_4(\mc A_n)$.

We have implemented this in Python; the code can be found at \url{https://github.com/mihirsinghal/resonance-arrangement/}. The code takes under 30 seconds to run on a personal laptop using Python 3, and gives the following formula for $b_4(\mc A_n)$:

\begin{multline*}
b_4(\mc A_n) = \f{1}{48} (2 \cdot 16^n - 18 \cdot 12^n + 31 \cdot 10^n + 33 \cdot 9^n - 119 \cdot 8^n + 77 \cdot 7^n \\ + 77 \cdot 6^n - 173 \cdot 5^n + 155 \cdot 4^n - 113 \cdot 3^n + 64 \cdot 2^n - 16).
\end{multline*}

Note that this result actually allows us to compute one fewer term $|G(p,9)|$ in \cref{sec:seqs}, which substantially improves the running time of that program.

\section*{Acknowledgments}
We would like to thank Haynes Miller and Roman Bezrukavnikov for insightful discussion. We would also like to thank Sunil Chebolu for bringing the problem to our attention. Thanks also to Antoine Deza and Lionel Pournin for thoughtful discussions. Thanks to Lukas K\"uhne for pointing out a typo in the interpolated coefficients of the polynomial $\chi(\mc A_9; t)$.

\printbibliography[heading=bibintoc]
\end{document}